\documentclass[review,onefignum,onetabnum]{siamart190516}
\usepackage{amsmath}
\usepackage{amssymb}
\usepackage{mathtools}

\def\R{{\mathbb R}}
\def\S{{\mathbb S}}

\def\TR{{\rm tr}}
\def\V{{\rm vec}}

\colorlet{BLUE}{blue}

\usepackage{lipsum}
\usepackage{amsfonts}
\usepackage{graphicx}
\usepackage{epstopdf}
\usepackage{algorithmic}
\usepackage{amsmath}
\usepackage{amssymb}
\usepackage{mathtools}

\usepackage{appendix}

\ifpdf
  \DeclareGraphicsExtensions{.eps,.pdf,.png,.jpg}
\else
  \DeclareGraphicsExtensions{.eps}
\fi

\newsiamremark{remark}{Remark}
\newsiamremark{hypothesis}{Hypothesis}
\crefname{hypothesis}{Hypothesis}{Hypotheses}
\newsiamthm{claim}{Claim}

\headers{Linear programming on the Stiefel manifold}{}

\title{Linear programming on the Stiefel manifold\thanks{Accepted by SIAM J. Optim.
\funding{This research was supported by the National Natural Science Foundation of China (Grant Nos. 12171021 and 11822103), and the Fundamental Research Funds for the Central Universities. The first author was also supported by the Academic Excellence Foundation of BUAA for PhD Students, and international joint doctoral education fund of Beihang University.}}}

\author{Mengmeng Song\thanks{School of Mathematical Sciences, Beihang University, Beijing 100191, People's Republic of China, 
\{songmengmeng, yxia\}@buaa.edu.cn.}
\and Yong Xia\footnotemark[2]
\thanks{Corresponding author. }
}

\usepackage{amsopn}

\ifpdf
\hypersetup{
  pdftitle={Linear programming on the Stiefel manifold},
  pdfauthor={}
}
\fi

\begin{document}

\maketitle

\begin{abstract}
Linear programming on the Stiefel manifold (LPS) is studied for the first time. It aims at minimizing a linear objective function over the set of all $p$-tuples of orthonormal vectors in $\R^n$ satisfying $k$ additional linear constraints.
	Despite the classical polynomial-time solvable case $k=0$, general (LPS) is NP-hard.
	According to the Shapiro-Barvinok-Pataki theorem, (LPS) admits an exact semidefinite programming (SDP) relaxation when $p(p+1)/2\le n-k$, which is tight when $p=1$. Surprisingly, we can greatly strengthen this sufficient exactness condition to
	$p\le n-k$, which covers the classical case $p\le n$ and $k=0$. Regarding (LPS) as a smooth nonlinear programming problem, we reveal a nice property that under the linear independence constraint qualification, the standard first- and  second-order {\it local} necessary optimality conditions are sufficient for {\it global} optimality when $p+1\le n-k$.
\end{abstract}

\begin{keywords}
{Linear programming, \and Stiefel manifold, \and  Quadratically constrained quadratic optimization, \and Semidefinite programming, \and Optimality conditions}
\end{keywords}

\begin{AMS}
  90C26, 90C22, 90C46,  90C20
\end{AMS}

\section{Introduction}
The Stiefel manifold  \cite{S35} is the set of all $p$-tuples of orthonormal vectors in $\R^n$ where $p\in\{1,\cdots,n\}$. We denote it by
\[
St_{n,p}=\{X\in\R^{n\times p}:~X^TX=I_p\},
\]
where $I_p$ is the identity matrix of order $p$. The two extreme cases $St_{n,1}$ and $St_{n,n}$ correspond to the unit sphere  and ``orthogonal group'', respectively. $St_{n,p}$ is connected if $p<n$, and has two connected components if $p=n$ \cite{Ra02}.

The fundamental linear optimization on the Stiefel manifold \cite{A09} reads as
\begin{equation}\label{LSM0}\tag{LS}
	\min_{X\in St_{n,p}} \TR(A_0X),
\end{equation}
where $A_0\in\R^{p\times n}$ and $\TR(\cdot)$ is the trace of matrix $(\cdot)$. \eqref{LSM0} can be globally solved by the singular value decomposition \cite{W02}, and the optimal value is equal to minus the sum of all singular values of $A_0$. 
In fact, \eqref{LSM0} dates back to von Neumann's trace inequality \cite{Neumann1937}. Furthermore, \eqref{LSM0} is equivalent to the minimization over the convex hull of $St_{n,p}$, which can be exactly characterized by linear matrix inequalities \cite{X14}.
Problem \eqref{LSM0} has an application in the projection onto $St_{n,p}$ \cite{Green1952,M02,Sc66}, more precisely,
\begin{equation}\label{qua-lin}
	\arg\min_{X\in St_{n,p}} \|A_0^T-X\|_F^2=\arg\min_{X\in St_{n,p}} \TR(-A_0X),
\end{equation}
where $\|\cdot\|_F$ is the Frobenius norm.

In this paper, we consider the following linear programming on the Stiefel manifold:
\begin{equation}\label{eq:LP_s}\tag{LPS}
	\begin{array}{cl}
		\min &\TR(A_0X)\\
		{\rm s.t.}&X\in E_{n,p,k}\triangleq\{X\in St_{n,p}:~ a_i^l \le\TR(A_iX)\le a_i^u,~i=1, \cdots, k\},
	\end{array}
\end{equation}
where $A_i\in\R^{p\times n}$, $a_i^l\in\R\cup\{-\infty\},$ and $a_i^u\in\R\cup\{+\infty\}$ for $i=1, \cdots, k$. We assume $E_{n,p,k}\neq\emptyset$.  
\eqref{eq:LP_s} encompasses several important problem instances, namely binary linear programming (Example~\ref{exam:1}), Hadamard conjecture (Example~\ref{exam2}), and linear sum assignment problem (Example~\ref{ex:linear_sum}), as special cases. Notably, the first one highlights NP-hardness of general \eqref{eq:LP_s}. \eqref{eq:LP_s} also has applications in a class of minimax Stiefel manifold optimization problems (Example~\ref{exam:6}). Moreover, \eqref{eq:LP_s} serves as the crucial subproblem in each iteration of both extended sequential linear programming (Section~\ref{subsec:2}) and extended Kelley's cutting plane method (Section~\ref{subsec:3}).
These two methods are utilized to address the constrained optimization problem under an additional Stiefel manifold constraint.

Note that \eqref{eq:LP_s} is not a manifold optimization problem except some special cases, such as \eqref{LSM0}.
The feasible set of \eqref{eq:LP_s}, $E_{n,p,k}$, may not be a manifold, even without inequality constraints and with only one equality constraint, see the discussion in Section \ref{se:5}.
Problem \eqref{eq:LP_s} is the linear programming problem on the Stiefel manifold, thus is a special class of constrained manifold optimization (CMO).
CMO refers to the constrained optimization problem with an additional manifold constraint, see \cite{Liu20,Weber23,Zhang14}.
Globally solving methods of CMO usually assume geodesically convexity of the feasible set, for example, \cite{Weber23}. However, even the simple case $E_{n,1,1}$, which corresponds to the intersection of a unit sphere and a linear inequality constraint, is not geodesically convex if it contains more than a half of the sphere but not the entire sphere \cite{Ud94}.
Other CMO-solving methods that can be applied to \eqref{eq:LP_s} cannot gurantee convergence to the globally optimal solution, such as the augmented Lagrangian method and the exact penalty method presented by Liu and Boumal  \cite{Liu20}.
If we regard \eqref{eq:LP_s} as an nonlinear programming (NLP) problem, it is also hard to globally solve due to the presence of $p(p+1)/2$ nonconvex quadratic constraints \cite{Luenberger1984}.

Fortunately, the concept of hidden convexity has unveiled novel prospects for globally solving nonconvex problems, attracting significant attention within the research community.
As defined in \cite{X20},
{\it hidden convexity} refers to the property of having a convex reformulation and allowing for a polynomial-time solving method.
Hidden convexity is often analyzed through convex relaxation exactness, strong duality, and optimality conditions, etc.


Based on the representation of $St_{n,p}$, \eqref{eq:LP_s} is a case of quadratic matrix optimization (QMP) problems \cite{Beck06}.
Beck \cite{Beck09} showed that QMP with a limited number of linear  constraints admits an exact semidefinite programming (SDP) relaxation by using the classical Shapiro-Barvinok-Pataki theorem \cite{Barvinok95,Pataki98,Shapiro82}. 
Furthermore, in that case,  the globally optimal solution of QMP can be extracted from that of the SDP relaxation.
It is a direct corollary that \eqref{eq:LP_s} with $p(p+1)/2\le n-k$ is hidden convex. We illustrate in Example \ref{exam:ex1} that this condition is tight when $p=1$ \cite{Beck09}, but fails to hold true in case $p\ge\sqrt{2n}$.
To our surprise, by carefully analyzing the structure of $E_{n,p,k}$, we can greatly strengthen the condition to $p\le n-k$. The new result implies  hidden convexity of \eqref{LSM0} by setting $p\le n$  and $k=0$.

For hidden convex optimization without any local nonglobal minimizer, it could be expected to characterize global optimality based on local information.
It is known that the first- and
second-order local necessary optimality conditions are sufficient for global optimality of
the homogeneous trust-region subproblem \cite{Ra02},
QMP problems arising from high-rank Burer-Monteiro factorizations \cite{BM03} of SDP problems \cite{Boumal2019,Wald20}, and so on.
For hidden convex optimization with at least one local nonglobal minimizer,  local optimality conditions are characterized case by case.
The nonconvex trust-region subproblem has at most one local nonglobal minimizer \cite{Mart94}, at which the first- and second-order sufficient optimality conditions are necessary  \cite{WX20}.
This property is extended to the homogeneous quadratic optimization with two quadratic constraints \cite{Song23} and the problem jointing trust-region subproblem with convex optimization \cite{S22}.

Our second contribution in this paper is to show that
under the linear independence constraint qualification (LICQ), the first- and  second-order  {\it local} necessary optimality conditions are sufficient for the {\it global} optimality  of \eqref{eq:LP_s}  when $p+1\le n-k$.

The remainder of this paper is organized as follows.
Section 2 gives plenty of applications of \eqref{eq:LP_s}.
Section 3 obtains an SDP relaxation of \eqref{eq:LP_s} via two approaches and an exactness result inheriting from that of QMP.
Section 4  establishes a significantly strengthened sufficient condition for the exactness of the same SDP relaxation.  Section 5 characterizes global optimality by the standard first- and second-order necessary local optimality conditions. Conclusions and open questions are presented in Section 6.

\noindent {\bf Notation.}
For an optimization problem $(P)$, denote its feasible region and optimal value by $\mathcal{F}(P)$ and  $v(P)$, respectively. The extreme points  of a given set $\Omega$ form the set $\mathcal{E}(\Omega)$.
Denote by $\S^{n}$ the set of symmetric matrices with dimension $n\times n$. For $A\in\S^{n}$,  $A\succeq (\preceq)\ 0$  means that $A$ is positive (negative) semidefinite.
For matrix $X\in\R^{n\times p}$, $X_i$ is its $i$-th column,  $X_{[i_1:i_2][j_1:j_2]}$ is the submatrix formed  by the entries in both $\{i_1,\cdots,i_2\}$-th rows  and $\{j_1,\cdots,j_2\}$-th columns. Notation $X_{[i_1:i_2]}$ is short for $X_{[i_1:i_2][i_1:i_2]}$. Denote by $\V(X)$ the vector obtained by stacking the columns of $X$ one underneath the
other. The Kronecker product of matrices $A$ and $B$ is denoted by $A\otimes B$.
Denote by $0_n$ the zero matrix with dimension $n\times n$.
Let $E^{ij}$  be the square matrix of a proper dimension with all entries being $0$ except the $i$-th row $j$-th column entry being $1$.
Let $e_i$ be the $i$-th column of the identity matrix $I_n$.
For $\lambda\in\R$, define $\lambda^+:={\rm max}\{0, \lambda\}$ and $\lambda^-:={\rm max}\{0, -\lambda\}$.

\section{Applications}
In this section, we present some applications of \eqref{eq:LP_s}.

\subsection{Classical \eqref{eq:LP_s} cases}\label{subsec:1}
We present classical optimization models that can be reformulated as \eqref{eq:LP_s}.
Our main focus is to investigate the computational complexity of globally solving \eqref{eq:LP_s} and to highlight the practical significance of \eqref{eq:LP_s} through a class of minimax problems.

\begin{example}\label{exam:1}Binary linear programming is a variant of linear programming in which the variables are additionally constrained to be either $-1$ or $1$. It is a classical NP-hard problem \cite{Ausiello99}.
	Let $A\in\R^{m\times n}$, $a \in\R^{n}$ and $ b\in\R^{m}$. The binary linear programming problem
	\begin{equation*}
		\begin{array}{cl}
			\min\limits_{x\in\{-1,1\}^n} &a^Tx\\
			{\rm s.t.}&Ax\le b
		\end{array}
	\end{equation*}
	is equivalent to the  following case of  \eqref{eq:LP_s} with $p=1$ and $k=n+m$:
	\begin{equation*}
		\begin{array}{cl}
			\min\limits_{X\in St_{n,1}} &\sqrt{n}a^TX\\
			{\rm s.t.}&\sqrt{n}AX\le b,\\
			&-\frac{1}{\sqrt{n}}\le X_i\le \frac{1}{\sqrt{n}},~i=1, \cdots, n,
		\end{array}
	\end{equation*}
	since
	\begin{equation}\label{reform:0}
		\{-1,1\}^n=\left\{\sqrt{n}X:~-\frac{1}{\sqrt{n}}\le X_i\le \frac{1}{\sqrt{n}},~i=1, \cdots, n,~X\in St_{n,1}\right\}.
	\end{equation}
\end{example}
The NP-hardness of the binary linear programming \cite{Ausiello99} in Example \ref{exam:1} already implies that of
\eqref{eq:LP_s}. For completeness, we also present a proof in the Appendix.
\begin{proposition}\label{pro:nphard}
General \eqref{eq:LP_s} is NP-hard.
\end{proposition}

Hadamard conjecture is another \eqref{eq:LP_s} case which suggests that verifying the feasibility of some cases of \eqref{eq:LP_s} is already challenging. 
\begin{example}\label{exam2}
	A Hadamard matrix is a square matrix with entries of either $+1$ or $-1$, and its rows are mutually orthogonal.
It has abundant applications in design theory, binary codes and so on, see \cite{HW78}.
The longstanding Hadamard conjecture focuses on the existence of Hadamard matrices for $n=1$, $n=2$, and $n$ being divisible by $4$, see \cite{YH97}.
	The existence of a Hadamard matrix with dimension $n$ amounts to the solvability of the following case of \eqref{eq:LP_s} with $p=n$ and $k=n^2$:
	\begin{equation*}
		\begin{array}{cl}
			\min\limits_{X\in St_{n,n}}& \TR(0_n X)\\
			{\rm s.t.} & -\frac{1}{\sqrt{n}} \le X_{ij}\le \frac{1}{\sqrt{n}},~i,j=1, \cdots, n,
		\end{array}
	\end{equation*}
for a similar reason as the equivalence \eqref{reform:0}. It is worth noting that the smallest order for which no Hadamard matrix is presently known is $668$  \cite{SWW05}.
\end{example}

There is also a classical \eqref{eq:LP_s} case which is easy-solving even with a large number of linear constraints.
\begin{example}\label{ex:linear_sum}
Linear sum assignment problem \cite{Burkard1999} is a fundamental combinatorial optimization problem that addresses the problem of assigning a set of agents to an equal number of tasks in a bijective manner, with the objective of minimizing the total cost of the assignment. Let $e=(1,\cdots,1)^T\in \R^n$. Its mathematical formulation
	\begin{equation}\label{assign}
		\begin{array}{cl}
			\min&  \TR(A_0X)  \\
		{\rm s.t.}& X\in \Pi_n \triangleq \{X\in \{0,1\}^{n\times n}:~ Xe=X^Te=e\} 
		\end{array}
	\end{equation}
is equivalent to the \eqref{eq:LP_s} case with $p=n$ and $k=n^2$:
\begin{equation*}
	\begin{array}{cl}
		\min\limits_{X\in St_{n,n}}&  \TR(A_0X)    \nonumber\\
		{\rm s.t.}& X_{ij}\ge 0,~i,j=1, \cdots, n, 
	\end{array}
\end{equation*}
since
	\begin{equation}\label{eq:equ}
	\Pi_n=\left\{X\in St_{n,n}: X_{ij}\ge 0,~i,j=1, \cdots, n \right\}.
\end{equation}

	It is further equivalent to the linear programming
	\begin{equation*}
		\begin{array}{cl}
			\min& \TR(A_0X)\\
			{\rm s.t.} & Xe=X^Te=e,\\
			&  X_{ij}\ge 0,~i,j=1, \cdots, n,
		\end{array}
	\end{equation*}
	by the well-known Birkhoff-von Neumann theorem  \cite{Birkhoff1946}.
\end{example}

In addition to aforementioned applications, \eqref{eq:LP_s} can also be applied to a class of minimax Stiefel manifold optimization problems. By utilizing the new result in this paper, we will reveal a new hidden convexity property of this class of optimization problems.
\begin{example}[Minimax Stiefel manifold optimization]\label{exam:6}
	For $M_i\in\R^{p\times n}$ and $c_i\in\R$ ($i=1,\cdots,m$), we consider the minimax problem
	\begin{equation}\label{minmax0}\tag{M-M}
		\min_{X\in E_{n,p,k}} \max \{\TR(M_1X)+c_1,\cdots,\TR(M_mX)+c_m\}.
	\end{equation}
When $k=0$, \eqref{minmax0} is the minimization of a piecewise linear nonsmooth convex function on the Stiefel manifold. This problem is fundamental in nonsmooth Stiefel manifold optimization and has various applications in machine learning \cite{Chen20,Zohrizadeh2019}.
Existing algorithms do not guarantee convergence to the globally optimal solution, such as the retraction-based proximal gradient method introduced by Chen et al. \cite{Chen20} and the sequential penalized relaxation method presented by Zohrizadeh et al. \cite{Zohrizadeh2019}.
	
	Problem \eqref{minmax0} can be equivalently transformed into the minimum of $m$ instances of \eqref{eq:LP_s}:
		\begin{equation*}
		\begin{array}{cl}
			\min\limits_{X\in E_{n,p,k}}& \TR(M_iX)+c_i\\
			{\rm s.t.} & \TR(M_iX)+c_i\ge \TR(M_jX)+c_j,~\forall j\neq i
		\end{array}
	\end{equation*}
	for $i=1,\cdots,m$.
By equivalently writting \eqref{minmax0} as
	 \begin{equation}\label{eq:t}
		\begin{array}{cl}
			\min\limits_{X\in E_{n,p,k},\, t\in\R} &t\\
			 \ {\rm s.t.} & \TR(M_iX)+c_i\le t,~i=1, \cdots, m,
	\end{array}
	\end{equation}
we will reveal hidden convexity of  \eqref{minmax0} similar to \eqref{eq:LP_s} in Section \ref{se:main}.
Here are some applications of \eqref{minmax0}.
\begin{itemize}
\item
Stiefel-constrained weighted maximin dispersion problem aims to find a point in the Stiefel manifold $St_{n,p}$ that maximizes the minimal weighted Euclidean
distance from given $m$ points. It reads as
\begin{equation}\label{eq:disp}
\max_{X\in St_{n,p}} \min_{i=1,\cdots,m}  w_i\|X-D_i\|_F^2
		=-\min_{X\in St_{n,p}} \max_{i=1,\cdots,m} -w_i\|X-D_i\|_F^2,
\end{equation}
	where $D_i\in\R^{n\times p}$, and $w_i>0$ for $i=1,\cdots,m$.
Problem \eqref{eq:disp} can be reformulated as a special case of \eqref{minmax0} similar to \eqref{qua-lin}. Wang and Xia \cite{WX16} studied ball-constrained weighted maximin dispersion problem. Based on their proof, one can easily derive that \eqref{eq:disp} with $p=1$ and $m\le n$ has an exact SDP relaxation. We will extend the result to larger $p$ cases in Section \ref{se:main}.
\item
Orthogonal dictionary learning is a fundamental problem in representation learning \cite{Bai2019,Xue2021}. To recover orthogonal dictionaries, one common approach is to consider the $\ell_1$-norm minimization problem which can be reformulated as a special case of \eqref{minmax0}
\begin{equation}\label{eq:orth}
	\min_{X\in St_{n,p}}\left\{  \|DX\|_1=\sum_{i=1}^m\sum_{j=1}^p  |(DX)_{ij}|=\max_{Y\in \{\pm1\}^{m\times p}} \sum_{i=1}^m\sum_{j=1}^p Y_{ij}(DX)_{ij}
	\right\},
\end{equation}
where $D\in \R^{m\times n}$. The $p=1$ case of \eqref{eq:orth}  is used by Bai et al. \cite{Bai2019} to provide a probable and approximate learned dictionary result.
\item Linear bottleneck assignment problem, introduced by Burkard and Derigs \cite{BD80}, is a variant of the linear assignment problem \eqref{assign}. It can be formulated as
\begin{equation}\label{eq:bottleneck}
	\min_{X\in \Pi_n} \max_{i,j=1,\cdots,n} C_{ij}X_{ij}.
\end{equation}
Problem \eqref{eq:bottleneck} seeks to minimize the maximum cost among all individual assignments, which is particularly useful in job assignment to parallel working machines.
Problem \eqref{eq:bottleneck} can be viewed as an application of \eqref{minmax0} by the equivalence \eqref{eq:equ}. Similar to the linear assignment problem \eqref{assign}, the linear bottleneck assignment problem can be solved efficiently in polynomial time due to its special structure \cite{Burkard2012}.
\end{itemize}
\end{example}

\subsection{Methods of solving CMO with \eqref{eq:LP_s} as the subproblem}
We demonstrate the applicability of \eqref{eq:LP_s} as the vital subproblem in both extended sequential linear programming and extended Kelley's cutting plane method for solving CMO.
\subsubsection{Extended sequential linear programming}\label{subsec:2}
Sequential linear programming (SLP) was developed by Griffith and Stewart \cite{Griffith61} for solving NLP problems. The SLP approach converts an NLP problem into a series of linear programming (LP) problems using the first-order Taylor expansion.

Consider the constrained optimization problem on the Stiefel manifold $St_{n,p}$:
\begin{equation}\label{eq:MC}
	\begin{array}{cl}
		\min\limits_{X\in St_{n,p}}&  f_0(X) \nonumber \\
		{\rm s.t.}& a_i^l\le f_i(X)\le a_i^u,\ i=1,\cdots, k,
	\end{array}
\end{equation}
where $k\ge 0$, $a_i^l\in\R\cup\{-\infty\},\  a_i^u\in\R\cup\{+\infty\}$ for $i=1, \cdots, k$, and $f_i$ is first-order differential for $i=0, 1, \cdots, k$.
To solve \eqref{eq:MC}, we extend SLP from the Euclidean space to the Stiefel manifold $St_{n,p}$. In iteration $t+1$, we solve the
\eqref{eq:LP_s}-subproblem
\begin{equation*}
	\begin{array}{cl}
		\min\limits_{X\in St_{n,p}}&  \nabla f_0(X^t)^T(X-X^t)  \\
		{\rm s.t.}& a_i^l\le \nabla f_i(X^t)^T(X-X^t)+ f_i(X^t)\le a_i^u,\ i=1,\cdots, k,
	\end{array}
\end{equation*}
where $X^t$ is obtained from iteration $t$ and $\nabla f$  represents the gradient of function $f$.
 It is worth noting that the convergence result of the original SLP typically requires an additional trust-region constraint such as
\begin{equation}\label{eq:tru} {\|X-X^t\|}_F^2\le\delta^t
\end{equation}
in each iteration, where $\delta^t>0$ represents the square of the trust-region radius.
For points $X$ and $X^t$ in $St_{n,p}$, \eqref{eq:tru} is equivalent to the linear constraint
$$\TR((X^{t})^TX)\ge p-\delta^t/2.$$
Thus, the involved subproblem remains an \eqref{eq:LP_s} case.

\subsubsection{Extended Kelley's cutting plane method}\label{subsec:3}Kelley's cutting plane method, introduced by Kelley \cite{Kelley1960}, is a classical approach for solving unconstrained convex problems that may be nonsmooth. If there exists a globally optimal point, the problem can be reformulated as a linear minimization problem over a compact convex set, by a similar approach used in \eqref{eq:t}.

In each iteration of Kelley's cutting plane method, a linear minimization problem is solved under the constraints of updated cutting planes. To maintain manageable complexity, nonbinding constraints are often dropped in each step \cite{Luenberger1984}. Kelley's method converges to the optimal solution under mild assumptions \cite{Kelley1960}.

We consider an extension of Kelley's cutting plane method to solve the problem under an additional Stiefel manifold constriant:
\begin{equation}\label{eq:ss}
	\begin{array}{cl}
		\min\limits_{X\in St_{n,p}\cap \Omega}&  \TR(A_0X), 
	\end{array}
\end{equation}
 where $\Omega$ is a closed convex set in $\mathbb{R}^{n\times p}$.
In each iteration, we keep $St_{n,p}$ and update the polyhedron that contains $\Omega$ using the same approach as in the original Kelley's cutting plane method. Consequently, each iteration requires solving a linear programming problem on $St_{n,p}$, which is an instance of \eqref{eq:LP_s}.
The convergence result remains, as the sequence of objective values will monotonically decrease in the same way as the original Kelley's cutting plane method.

A more recent advancement was made by Drori and Teboulle \cite{Drori16}, who presented a variant of Kelley's cutting plane method with an optimal convergence rate. It is unknown whether this optimal  method can be extended to solve \eqref{eq:ss}.

\section{ SDP relaxation}
We first present a direct SDP relaxation for \eqref{eq:LP_s} in Section \ref{subsec3}. In Section \ref{subsec4}, the same relaxation is reobtained based on the QMP reformulation of \eqref{eq:LP_s}. In Section \ref{subsec5}, an exactness result of the SDP relaxation is obtained, inheriting from the relevant result of QMP.
\subsection{The first approach}\label{subsec3}
The nonconvexity of \eqref{eq:LP_s} comes from the quadratic equality constraints due to $St_{n,p}$. We start from the geometric characteristics of $St_{n,p}$.
\begin{lemma}\label{lem:1}{\rm(\cite[Lemma 1]{X14})}
	The convex hull of $St_{n,p}$ is
	\begin{equation}\label{cvxhul}
		\{X\in\R^{n\times p}: ~X^TX\preceq I_p\}.
	\end{equation}
\end{lemma}
Replacing $St_{n,p}$ with \eqref{cvxhul} yields the conic convex optimization relaxation:
\begin{equation}\label{eq:LPS-R}\tag{CR}
	\begin{array}{cl}
		\min\limits_{X\in\R^{n\times p}} &\TR(A_0X)\\
		 {\rm s.t.} &a_i^l\le \TR(A_iX)\le a_i^u,~i=1, \cdots, k,\\
	&X^TX\preceq I_p.
	\end{array}
\end{equation}

For $X\in\R^{n\times p}$, define
\begin{equation}\label{eq:YX}
	Y =\bmatrix
	I_n & X\\
	X^T & I_p
	\endbmatrix \in \S^{n+p }.
\end{equation}
According to the following congruent transformation
\begin{equation}\label{eq:YxX}
	Y =
	\bmatrix
	I_n & 0\\
	X^T & I_p
	\endbmatrix
	\bmatrix
	I_n & 0\\
	0 & I_p-X^TX
	\endbmatrix
	\bmatrix
	I_n & X\\
	0 & I_p
	\endbmatrix,
\end{equation}
it holds that
\begin{equation}\label{eq:equi0}
	X^TX\preceq I_p \Longleftrightarrow Y\succeq0.
\end{equation}
By \eqref{eq:YX} and \eqref{eq:equi0}, the conic convex optimization  \eqref{eq:LPS-R}  is equivalent to the following
SDP problem:
\begin{equation}\label{eq:LPS-SDP3}\tag{SDR}
	\begin{array}{cl}
		\min\limits_{Y\in\S^{n+p}}  & \TR(B_0Y)\\
	{\rm s.t.} & a_i^l\le\TR(B_iY)\le a_i^u,~i=1, \cdots, k,\\
		&Y_{[1:n]}=I_n,\ Y_{[n+1:n+p]}=I_p,\\
		&Y\succeq 0,
	\end{array}
\end{equation}
where
\begin{equation}\label{eq:Bi}
	B_i=\frac{1}{2}\bmatrix
	0_n & A_i^T\\
	A_i & 0_p
	\endbmatrix
	{\rm\ for\ } i=0, 1, \cdots, k.
\end{equation}
We present the relationships among \eqref{eq:LP_s}, \eqref{eq:LPS-R}, and \eqref{eq:LPS-SDP3} in the following, and put the proofs in the Appendix.
\begin{lemma} \label{lem:2}
	Suppose $\mathcal{F}\eqref{eq:LPS-SDP3}\neq \emptyset$. Then
	\eqref{eq:LPS-SDP3} has an optimal solution and
	$v\eqref{eq:LPS-R}=v\eqref{eq:LPS-SDP3}$. Moreover, $X$ solves \eqref{eq:LPS-R} if and only if $Y$ defined in \eqref{eq:YX} solves \eqref{eq:LPS-SDP3}.
\end{lemma}

\begin{proposition}\label{prop:1}
	If \eqref{eq:LPS-SDP3} has an optimal solution of rank at most $n$, then \eqref{eq:LPS-SDP3} is exact, i.e., $v\eqref{eq:LPS-SDP3}=v\eqref{eq:LP_s}$.
\end{proposition}

\subsection{The second approach}\label{subsec4}
In this subsection, we reestablish the SDP relaxation \eqref{eq:LPS-SDP3} by first reformulating \eqref{eq:LP_s} as a case of QMP \cite{Beck06} and then employing  its SDP relaxation presented in \cite{Beck09}.
QMP is the quadratically constrained quadratic optimization in terms of a matrix variable,
	\begin{equation}\label{eq:QMP}\tag{QMP}
		\begin{array}{cl}
			\min\limits_{X\in\R^{n\times p}} & q_0(X)\\
			{\rm s.t.} & q_i(X)\le (=) 0,~ i=1, \cdots, m,
		\end{array}
	\end{equation}
where each involved quadratic function is of the form
$$q_i(X)=\TR(XQ_iX^T)+2\TR(A_iX)+a_i,$$
with $Q_i\in\R^{p\times p}$, $ A_i\in\R^{p\times n}$, and $a_i\in\R$ for $i=0, \cdots, m$.

 Beck \cite{Beck09} equivalently homogenized QMP as
 \begin{equation*}
 	\begin{array}{cl}
 		\min\limits_{X\in\R^{n\times p},\ Z\in\R^{p\times p}} & 	\TR\left(\bmatrix
 		\frac{a_0}{n}I_n & A_0^T\\
 		A_0 & Q_0
 		\endbmatrix\bmatrix
 		I_n & X\\
 		X^T & Z
 		\endbmatrix\right)\\
 		{\rm s.t.} & 	\TR\left(\bmatrix
 		\frac{a_i}{n}I_n & A_i^T\\
 		A_i  & Q_i
 		\endbmatrix\bmatrix
 		I_n & X\\
 		X^T & Z
 		\endbmatrix\right)\le (=) 0,\ i=1, \cdots, m,\\
 		& X^TX=Z.
 	\end{array}
 \end{equation*}
Denote
$$C_i=\begin{bmatrix}
	\frac{a_i}{n}I_n & A_i^T\\
	A_i & Q_i
\end{bmatrix} {\rm\ for\ } i=0, \cdots, m,\ {\rm and}~ Y=\begin{bmatrix}
	I_n & X\\
	X^T & Z
\end{bmatrix}.$$
According to Schur's complement theorem, we have
\begin{equation*}
	\begin{aligned}
		& X^TX=Z \Longleftrightarrow Y\succeq 0 {\rm\ and\ } {\rm rank}(Y)=n.
	\end{aligned}
\end{equation*}

Dropping the rank constraint leads to the following SDP
relaxation \cite{Beck09}
 \begin{equation}\label{eq:qapc2}
	\begin{array}{cl}
		\min\limits_{Y\in\R^{(n+p)\times (n+p)}} & 	\TR\left(C_0Y\right)\\
		{\rm s.t.} & \TR\left(C_iY\right)\le (=) 0,\ i=1, \cdots, m,\\
		&Y_{[1:n]}=I_n,\\
		&Y\succeq 0.
	\end{array}
\end{equation}
The following exactness result was presented based on the classical Shapiro-Barvinok-Pataki theorem \cite{Barvinok95,Pataki98,Shapiro82}, which can be found in Remark \ref{re:ori}.
\begin{lemma}{\rm(\cite[Theorem 2.2]{Beck09})}\label{le:Beck}
If problem \eqref{eq:qapc2} is solvable and $m\le n$, then problem \eqref{eq:QMP} is solvable and ${\rm val}\eqref{eq:qapc2}={\rm val}\eqref{eq:QMP}$.
	\end{lemma}

Problem \eqref{eq:LP_s} reads as a QMP case
\begin{equation}\label{eq:AQMP}
	\begin{array}{cl}
		\min\limits_{X\in\R^{n\times p}} &\TR(A_0X)\\
		{\rm s.t.} & a_i^l\le \TR(A_iX)\le a_i^u,~i=1, \cdots, k,\\
		&\TR(XE^{ij}X^T)=\delta_{ij},~1\le i\le j\le p,
	\end{array}
\end{equation}
where $\delta_{ij}=e_i^Te_j$ is the Kronecker delta. Applying the relaxation \eqref{eq:qapc2} to \eqref{eq:AQMP} gives
\begin{equation}\label{eq:AQMP1}
	\begin{array}{cl}
		\min\limits_{Y\in\R^{{(n+p)}\times {(n+p)}}} &\TR(B_0Y)\\
		{\rm s.t.} & a_i^l\le \TR(B_iY)\le a_i^u,~i=1, \cdots, k,\\
		&\TR( E^{ij}Y_{[n+1:n+p]})=\delta_{ij},~1\le i\le j\le p \Leftrightarrow Y_{[n+1:n+p]}=I_p,\\
		&Y_{[1:n]}=I_n,\\
		&Y\succeq0,
	\end{array}
\end{equation}
where $B_0, \cdots, B_k$ are defined in \eqref{eq:Bi}. 
The relaxation \eqref{eq:AQMP1} coincides with our relaxation \eqref{eq:LPS-SDP3}.
	\subsection{Exactness}\label{subsec5}
Under the assumption that 
\begin{equation}\label{bd:k0}
	a_i^l=a_i^u {\rm\ or\ } a_i^l=-\infty {\rm\ or\ }a_i^u=+\infty,\ \forall i=1, \cdots, k,
\end{equation}
problem \eqref{eq:AQMP} is a QMP case with
$k+p(p+1)/2$ constraints.  
As a direct corollary of Lemma \ref{le:Beck}, \eqref{eq:AQMP1} (or equivalently, \eqref{eq:LPS-SDP3})
is exact when
\begin{equation}\label{bd:k1}
\frac{p(p+1)}{2} \le n- k.
\end{equation}
Note that the Shapiro-Barvinok-Pataki theorem \cite{Barvinok95, Pataki98, Shapiro82} plays a key role in the exactness proof of Lemma \ref{le:Beck}.
To get rid of the assumption \eqref{bd:k0}, we slightly extend the Shapiro-Barvinok-Pataki theorem to the bilateral case.
\begin{lemma}\label{le:pataki}
	Let $n_1>n_2$.  Consider $\mathcal{B}:=\{Y\in \S^{n_1}: \ Y\succeq0,\ a_i^l\le \TR(B_iY)\le\ a_i^u,~i=1, \cdots, m\}$ and assume $\mathcal{B}\neq\emptyset$. If $m\le (n_2+2)(n_2+1)/2-1$, then
	\begin{equation}\label{eq:Ye}{\rm rank}(Y^e)\le n_2
\end{equation}
		 holds  for all $Y^e\in \mathcal{E}(\mathcal{B})$.
\end{lemma}
\begin{remark}\label{re:ori}
The original Shapiro-Barvinok-Pataki theorem can be seen as a specific case of Lemma \ref{le:pataki} when the additional condition \eqref{bd:k0} is satisfied. 
Expanding upon this, it is straightforward to get Lemma \ref{le:pataki}.
To illustrate, let us consider an extreme point $Y^e$ of the set $\mathcal{B}$, i.e., $Y^e\in \mathcal{E}(\mathcal{B})$. 
Then, $Y^e$ remains an extreme point of the following set 
 \begin{equation}\label{eq:pat}
\{Y\in \S^{n_1}: \ Y\succeq0,\ \TR(B_iY)= \TR(B_iY^e),~i=1, \cdots, m\}.
\end{equation}
Applying the original Shapiro-Barvinok-Pataki theorem to the set \eqref{eq:pat} yields the result \eqref{eq:Ye}.
\end{remark}

As mentioned above, the exactness result of \eqref{eq:LPS-SDP3} under conditions \eqref{bd:k0} and \eqref{bd:k1} has been established. In the following, we provide the result for cases beyond \eqref{bd:k0} based on Lemma \ref{le:pataki}. For completeness, the proof is presented in the Appendix.
\begin{proposition}\label{le:beck09}
	Suppose that $\mathcal{F}\eqref{eq:LPS-SDP3}\neq \emptyset$ and \eqref{bd:k1} hold,
then
$v\eqref{eq:LPS-SDP3}=v\eqref{eq:LP_s}$.
\end{proposition}

\section{Strengthened exactness result of the SDP relaxation}\label{se:main}
In this section, we greatly strengthen the condition for the exactness of \eqref{eq:LPS-SDP3} presented in Proposition \ref{le:beck09}.  Then, we
show the tightness of our new  condition by two examples.

On the one hand, we use the following example to show that condition \eqref{bd:k1} required in Proposition \ref{le:beck09} is already tight for the exactness of \eqref{eq:LPS-SDP3} when $p=1$.
\begin{example}\label{exam:ex1}
	Consider an instance of \eqref{eq:LP_s} with $p=1$, $n=2$, and $k=2$:
	\begin{equation}\label{eq:LP_s0}
		\begin{array}{cl}
			\min\limits_{X\in St_{2,1}} &-x_1-x_2\\
			{\rm s.t.}& x_{1}\le 0,\\
			& x_{2}\le 0.
		\end{array}
	\end{equation}
	The relaxation \eqref{eq:LPS-R} is not exact  as one can verify that
\[
	v\eqref{eq:LP_s0}=1>0\ =\ \min_{x\in\R^{2},\  x_{1}\le 0,\ x_{2}\le 0,\ x^Tx\le1}-x_1-x_2.
	\]
	By Lemma  \ref{lem:2}, \eqref{eq:LPS-SDP3} is not  exact.
\end{example}

On the other hand, when $p\ge \sqrt{2n}$, the condition \eqref{bd:k1} stated in Proposition \ref{le:beck09} is not satisfied, even when $k=0$. Note that Lemmas \ref{lem:1} and \ref{lem:2} imply that the SDP relaxation of \eqref{LSM0} provides an exact solution. This indicates that \eqref{eq:LPS-SDP3} is exact for all \eqref{eq:LP_s} instances with $p\le n, k=0$. Therefore, it is reasonable to anticipate an improvement to Proposition \ref{le:beck09}.

We now present our strengthened result.
\begin{theorem}\label{th:1}
	When $p\le n-k$,
	$v\eqref{eq:LPS-SDP3}=v\eqref{eq:LP_s}$.
\end{theorem}
\begin{proof}
Without loss of generalization, we assume that $\mathcal{F}\eqref{eq:LPS-SDP3}\neq \emptyset$. Otherwise, $v\eqref{eq:LPS-SDP3}=v\eqref{eq:LP_s}=+\infty$. Since the objective function of \eqref{eq:LPS-SDP3} is linear, there exists an optimal solution of \eqref{eq:LPS-SDP3} satisfying
$Y\in \mathcal{E}(\mathcal{F}\eqref{eq:LPS-SDP3})$.
	According to Proposition \ref{prop:1}, it is sufficient to show that ${\rm rank}(Y)\le n$.
	Suppose, on the contrary, ${\rm rank}(Y)=n+s$ for some $s\in\{1,\cdots, p\}$.
	
The remaining proof consists of four steps.

	{\bf Step 1: Decomposition of matrix $Y$ based on its rank.}
	Since $Y$ is feasible for \eqref{eq:LPS-SDP3}, we can rewrite $Y$ as in \eqref{eq:YX}
	for some $X\in\R^{n\times p}$. According to \eqref{eq:equi0} and $Y\succeq 0$, we have
	\begin{equation}
		I_p-X^TX\succeq0.\label{main:0}
	\end{equation}
	It follows from
	\begin{equation}\label{eq:62}
		{\rm rank}(Y)={\rm rank}(I_n)+{\rm rank}(I_p-X^TX)
		= n+{\rm rank}(I_p-X^TX)
	\end{equation}
 and assumption ${\rm rank}(Y)=n+s$ that
	\begin{equation}
		\ {\rm rank}(I_p-X^TX)=s.\label{main:1}
	\end{equation}
	Based on \eqref{main:0}-\eqref{main:1},
	there exists a  full-column-rank matrix $C\in\R^{p\times s}$ such that
	$I_p-X^TX= CC^T$.
	Define
	\begin{equation}\label{eq:U}
		U=\bmatrix
		I_n & 0_{n\times s}\\
		X^T & C
		\endbmatrix\in\R^{(n+p)\times(n+s)}.
	\end{equation}
 Then, $U$ is of full column rank,  and
	\begin{equation*}
		Y=
		UU^T.
	\end{equation*}
	
	{\bf Step 2:  Formulation of a linear system in condition that if the linear system has a nonzero solution, then $Y\notin \mathcal{E}(\mathcal{F}\eqref{eq:LPS-SDP3})$ holds.}
	Consider the following linear system in terms of matrix $D$:
	\begin{eqnarray}
		&&\TR(B_iUDU^T)= 0,~i=1, \cdots, k,\label{eq:D1}\\
		&&{UDU^T}_{[1:n]}=0_n,\label{eq:D2}\\ &&{UDU^T}_{[n+1:n+p]}=0_p,\label{eq:D3}\\
		&&D\in \S^{n+s}.\label{eq:D4}
	\end{eqnarray}
	If \eqref{eq:D1}-\eqref{eq:D4} has a nonzero solution $D$,  then define
	\begin{equation}\label{Yepsn}
		Y(\epsilon):=Y+\epsilon UDU^T=U(I+\epsilon D)U^T,~\epsilon\in\R.
	\end{equation}
	Since $Y$ is feasible for \eqref{eq:LPS-SDP3},
	for any $\epsilon\in\R$, we have
	\begin{eqnarray}
		&& \TR(B_iY(\epsilon))=\TR(B_iU(I+\epsilon D)U^T)= \TR(B_iY)\in [a_i^l, a_i^u],~i=1, \cdots, k,\label{eq:111}\\
		&&Y(\epsilon)_{[1:n]}={U(I+\epsilon D)U^T}_{[1:n]}={UU^T}_{[1:n]}=I_n,\label{eq:112}\\ &&Y(\epsilon)_{[n+1:n+p]}={U(I+\epsilon D)U^T}_{[n+1:n+p]}={UU^T}_{[n+1:n+p]}=I_p,\label{eq:113}\\
		&&Y(\epsilon)=U(I+\epsilon D)U^T=Y(\epsilon)^T.\label{eq:114}
	\end{eqnarray}
	Moreover, by the definition of  $Y(\epsilon)$,  for any $\epsilon$ such that $|\epsilon|$ is sufficiently small, it holds that
	\begin{equation}
		Y(\epsilon)=U(I+\epsilon D)U^T\succeq 0.\label{eq:115}
	\end{equation}
	Since $U$ has the full-column rank,  it follows from $D\neq 0$ that $UDU^T\neq0$.
 By \eqref{Yepsn}-\eqref{eq:115}, $Y(\epsilon)\neq Y$ is feasible for \eqref{eq:LPS-SDP3} when $\epsilon\neq0$ and $|\epsilon|$ is sufficiently small.
Combining with $Y=(Y(\epsilon)+Y(-\epsilon))/2$, we obtain a contradiction that $Y\not\in \mathcal{E}(\mathcal{F}\eqref{eq:LPS-SDP3})$.
	
{\bf Step 3:  Analysis of the coefficient matrix of the linear system \eqref{eq:D1}-\eqref{eq:D4}.}
Now, it suffices to show that the linear system \eqref{eq:D1}-\eqref{eq:D4} has a nonzero solution.
	Firstly, we  rewrite \eqref{eq:D2} as
	\begin{equation}
		\TR(E^{ij}UDU^T)=0,\ \forall \ 1\le i, \ j\le n.\label{eq:D22}
	\end{equation}
	By the definition of $U$ \eqref{eq:U}, one  can  verify that
	\[
	U^TE^{ij}U=E^{ij},~1\le i,\ j\le n,
	\]
	and hence
	\[
	D_{ij}=\TR(E^{ij}D)=\TR(U^TE^{ij}UD)=\TR(E^{ij}UDU^T),~1\le i,\ j\le n.
	\]
	 Thus,  equations \eqref{eq:D22}  are equivalent to
	\begin{equation}
		D_{ij}=0,\ \forall \ 1\le i, j\le n.\label{eq:D222}
	\end{equation}
	Secondly, \eqref{eq:D3} holds if and only if
	\begin{equation}
		\TR(U^TE^{(n+i)(n+j)}UD)=\TR(E^{(n+i)(n+j)}UDU^T)=0,\ \forall \ 1\le i, j\le p.\label{eq:D32}
	\end{equation}
	Let $x_i\in\R^n $ be the $i$-th column of $X$ and $c_i\in\R^s$ be the $i$-th column of  $C^T$ for $i=1, \cdots, p$. Then one can verify that
	\begin{equation}\label{eq:2}
		U^TE^{(n+i)(n+j)}U=\bmatrix
		x_ix_j^T & x_ic_j^T\\
		c_ix_j^T& c_ic_j^T
		\endbmatrix,~1\le i, j\le p.
	\end{equation}
	Substituting \eqref{eq:2} into \eqref{eq:D32} yields
	\[
	\TR\left(\bmatrix
	x_ix_j^T & x_ic_j^T\\
	c_ix_j^T& c_ic_j^T
	\endbmatrix D\right)=0,\ \forall 1\le i, j\le p,
	\]
	which are equivalent to
	\begin{eqnarray}
		0=&&\ 2\TR\left(\bmatrix
		x_ix_j^T & x_ic_j^T\\
		c_ix_j^T& c_ic_j^T
		\endbmatrix D\right) \nonumber\\
		=&&\ 2\TR\left(\bmatrix
		0_n & x_ic_j^T\\
		c_ix_j^T& c_ic_j^T
		\endbmatrix D\right) ~~~~({\rm since}~ \eqref{eq:D222})\nonumber\\
		=&&\ \TR\left(\bmatrix
		0_n & x_ic_j^T\\
		c_ix_j^T& c_ic_j^T
		\endbmatrix D\right)
		+\TR\left(D^T\bmatrix
		0_n & x_ic_j^T\\
		c_ix_j^T& c_ic_j^T
		\endbmatrix^T \right) \nonumber\\
		=&&\ \TR\left(\bmatrix
		0_n & x_ic_j^T\\
		c_ix_j^T& c_ic_j^T
		\endbmatrix D\right)
		+\TR\left(\bmatrix
		0_n & x_jc_i^T\\
		c_jx_i^T& c_jc_i^T
		\endbmatrix D\right) ~~~~({\rm since}~ \eqref{eq:D4})\nonumber\\
		=&&\ \TR\left(\bmatrix
		0_n & x_ic_j^T\\
		c_jx_i^T& c_ic_j^T
		\endbmatrix D\right)
		+\TR\left(\bmatrix
		0_n & x_jc_i^T\\
		c_ix_j^T& c_jc_i^T
		\endbmatrix D\right),~\forall 1\le i, j\le p.\label{eq:34}
	\end{eqnarray}
We denote $D\in \S^{n+s}$ as
	\begin{eqnarray}
		&&D=\bmatrix
		D_{11}& D_{12}\\
		D_{21}& D_{22}
		\endbmatrix,\nonumber
	\end{eqnarray}
	where $D_{11}\in\S^{n }$, $D_{12}\in\R^{n\times s}$, $D_{21}=D_{12}^T$, $D_{22}\in\S^{s }$. Let ${\rm vt}(\cdot)=(\V(\cdot))^T$.
	Define
	\begin{eqnarray}
		d&=&\bmatrix
		{\rm vt}(D_{11})&{\rm vt}(D_{12})&
		{\rm vt}(D_{21})& {\rm vt}(D_{22})
		\endbmatrix^T\in\R^{(n+s)^2}, \nonumber\\
		M_1&=&\bmatrix
		\cdots&\cdots&\cdots&\cdots\\
		{\rm vt}(0_n)& {\rm vt}(x_ic_j^T) &
		{\rm vt}(c_jx_i^T )& {\rm vt}(c_ic_j^T)\\
		\cdots&\cdots&\cdots&\cdots
		\endbmatrix_{1\le i, j\le p} \in\R^{ p^2\times(n+s)^2}\nonumber,\\
		M_2&=&\bmatrix
		\cdots&\cdots&\cdots&\cdots\\
		{\rm vt}(0_n)& {\rm vt}(x_jc_i^T) &
		{\rm vt}(c_ix_j^T)& {\rm vt}(c_jc_i^T)\\
		\cdots&\cdots&\cdots&\cdots
		\endbmatrix_{1\le i, j\le p} \in\R^{ p^2\times(n+s)^2}\nonumber.
	\end{eqnarray}
	Since $\TR(AB^T)={\rm vt}(A)({\rm vt}(B))^T$ holds for any matrices $A$ and $B$,  we can reformulate \eqref{eq:34} as
	\begin{equation}\label{eq:M12}
		(M_1+M_2)d=0.
	\end{equation}
	As the rows of $M_1$ and $M_2$ are the same except for the sort order, we have
	\begin{equation}\label{eq:RM}
		{\rm rank}(M_1)={\rm rank}(M_2)={\rm rank}(M_1+M_2).
	\end{equation}
	For any  $1\le i\le p$, define the $i$-th submatrix of $M_1$ as
	\[
	\begin{aligned}
		M_1^i=&\bmatrix
		\cdots&\cdots&\cdots&\cdots\\
		\V(0_n)& \V(x_ic_j^T) &
		\V(c_jx_i^T )& \V(c_ic_j^T)\\
		\cdots&\cdots&\cdots&\cdots
		\endbmatrix_{1\le j\le p}\in\R^{ p\times(n+s)^2}.
	\end{aligned}
	\]
	Since rank$([c_{1}~c_{2} ~ \cdots ~ c_{p}]^T)=$rank$(C)=s$, we have
	\begin{equation}
		{\rm rank}(M_1^i)\le s,~1\le i\le p. \label{eq:RM0}
	\end{equation}
	Then, it follows from \eqref{eq:RM}-\eqref{eq:RM0} that
	\begin{equation}
		{\rm rank}(M_1+M_2)={\rm rank}(M_1)\le \sum_{i=1}^p  {\rm rank}(M_1^i)\le ps. \label{eq:RM1}
	\end{equation}
	
	{\bf Step 4:  Verification that the linear system has a nonzero solution.}
	In sum,  the linear equations  \eqref{eq:D2}-\eqref{eq:D3}  are equivalent to \eqref{eq:D222} and   \eqref{eq:M12}, respectively. The rank of the coefficient matrix of \eqref{eq:D222} is $n^2$, and that of \eqref{eq:M12} is no more than $ps$.
	The symmetry condition \eqref{eq:D4} is equivalent to a linear system with $(n+s)(n+s-1)/2$ equalities, where $n(n-1)/2$ of them are already contained in \eqref{eq:D222}.
	
 To sum up the above analysis,  the whole system \eqref{eq:D1}-\eqref{eq:D4} amounts to a linear system in terms of $D\in\R^{(n+s)^2}$ where the rank $R$ of the coefficient matrix satisfies:
	\begin{eqnarray}
		R\le\  &&k+n^2+ps+\frac{(n+s)(n+s-1)}{2}-\frac{n(n-1)}{2}\nonumber\\
		=&&\ k+\frac{n(n+1)}{2}+ps+\frac{(n+s)(n+s-1)}{2}.\nonumber
	\end{eqnarray}
	Under the assumption $p\le n-k$, we have
	\begin{equation*}
		\begin{aligned}
			(n+s)^2-R=\
			&\frac{(n+s)(n+s+1)}{2}-\left[k+\frac{n(n+1)}{2}+ps\right]\\
			\ge\ &\frac{s(s+1)}{2}+(n-p)(s-1)\\
			\ge\ &1, ~~~~ ~({\rm since}\ s\ge 1,~p\le n).
		\end{aligned}
	\end{equation*}
	Consequently, \eqref{eq:D1}-\eqref{eq:D4} must have a nonzero solution. The proof is complete.
\end{proof}

\begin{remark}\label{th:rem:1}
	According to the proofs of Proposition \ref{prop:1} and Theorem \ref{th:1}, when $p\le n-k$,
	\eqref{eq:LPS-SDP3} has an optimal solution of rank $n$, and \eqref{eq:LPS-R} has an optimal solution in $E_{n,p,k}$.
\end{remark}
\begin{remark}
	Suppose $p\le n-k$.
	The proof of Theorem \ref{th:1} indicates a rank reduction algorithm to obtain an optimal rank-$n$ solution of \eqref{eq:LPS-SDP3} similarly as in \cite{Beck06}. Let $Y$ be an optimal solution to \eqref{eq:LPS-SDP3}. If rank$(Y)>n$, we first
	solve \eqref{eq:D1}-\eqref{eq:D4} to get a nonzero $D$ and define $Y(\epsilon)$ as in \eqref{Yepsn}.
	Since $Y$ is optimal  and  $Y(\epsilon)$  is feasible for sufficiently small $|\epsilon|$, we have
	\[\frac{\TR(B_0Y(\epsilon))+\TR(B_0Y(-\epsilon))}{2}=\TR(B_0Y)\le\TR(B_0Y(\epsilon)).\]
Thus,	
\begin{equation}\label{eq:optimal}
	\TR(B_0Y)=\TR(B_0Y(\epsilon))
	\end{equation}
holds, if $|\epsilon|$ is sufficiently small.

	Then we can find a suitable $\widetilde{\epsilon}$ satisfying \eqref{eq:115} and  rank$(Y(\widetilde{\epsilon}))\le$ rank$(Y)-1$. Setting $Y:=Y(\widetilde{\epsilon})$, $Y$ remains optimal by \eqref{eq:optimal}. This procedure is repeated until  rank$(Y)=n$.
\end{remark}

When $p=1$, Theorem \ref{th:1} coincides with Proposition \ref{le:beck09}, and the condition is tight as shown in Example \ref{exam:ex1}. When $p>1$,
Theorem \ref{th:1} strictly improves Proposition \ref{le:beck09}.
We illustrate the tightness of condition $p\le n-k$ presented in Theorem \ref{th:1} by two more examples.
\begin{example}\label{exam:ex2}
	Consider an instance of \eqref{eq:LP_s} with $p=2$, $n=3$, and $k=2$:
		\begin{equation}\label{eq:LP_s1}
		\begin{array}{cl}
			\min\limits_{X\in St_{3, 2}}&X_{3,2}\\
			{\rm s.t.}& X_{1,1}=X_{2,1}=0.
		\end{array}
	\end{equation}
	One can easily verify that $v\eqref{eq:LP_s1}=0$ as
	$$\mathcal{F}\eqref{eq:LP_s1}=\{X\in\R^{3\times 2}: X_{1,1}=X_{2,1}=X_{3,2}=0,\  X_{3,1}=\pm1,\  X_{1,2}^2+X_{2,2}^2=1\}.$$
	The relaxation \eqref{eq:LPS-R}  is given by
	\begin{equation}\label{eq:LP_s1d}
		\begin{array}{cl}
			\min\limits_{X\in\R^{3\times 2}}&X_{3,2}\\
			{\rm s.t.}& X_{1,1}=X_{2,1}=0,\\
			&X^TX\preceq I_2.
		\end{array}
	\end{equation}
	where
	\begin{equation*}
		\bar X=\bmatrix
		0&0\\
		0&0\\
		0&-1
		\endbmatrix
	\end{equation*}
 is feasible.
It follows that
	$$v\eqref{eq:LP_s1d}\le \bar X_{3,2}=-1<0=v\eqref{eq:LP_s1}.$$
	That is, \eqref{eq:LP_s1d} is not an  exact relaxation for \eqref{eq:LP_s1},
	so is \eqref{eq:LPS-SDP3} due to
	Lemma  \ref{lem:2}.
\end{example}
\begin{example}\label{exam:ex3}
	Consider an instance of \eqref{eq:LP_s} with $p=3$, $n=3$, and  $k=1$:
	\begin{equation}\label{eq:LP_s2}
		\begin{array}{cl}
			\min\limits_{X\in St_{3, 3}}&X_{2,2}+X_{3,3}\\
			{\rm s.t.}& X_{1,1}=0.
		\end{array}
	\end{equation}
	The \eqref{eq:LPS-R} relaxation
		\begin{equation}\label{eq:LP_s2d}
		\begin{array}{cl}
			\min\limits_{X\in\R^{3\times 3}}&X_{2,2}+X_{3,3}\\
			{\rm s.t.}& X_{1,1}=0,\\
			&X^TX \preceq I_3.
		\end{array}
	\end{equation}
	has a feasible solution
	\begin{equation*}
		\bar X=\bmatrix
		0 & 0 & 0\\
		0 & -1 & 0\\
		0 & 0 & -1
		\endbmatrix.
	\end{equation*}
	Then it holds that $v\eqref{eq:LP_s2d}\le-2$.
	On the other hand, for any $X$ satisfying
	$X^TX\preceq I_3$, we have $X_{2,2}^2\le 1$ and $X_{3,3}^2\le 1$, which imply that
	$v\eqref{eq:LP_s2d}\ge-2$.  Therefore, $v\eqref{eq:LP_s2d}=-2$ and
	$\bar X$ is an optimal solution of \eqref{eq:LP_s2d}.
	
	It is not difficult to verify that any
	\begin{equation*}
		X=\bmatrix
		 0 & X_{1,2} & X_{1,3}\\
		X_{2,1} & -1 & X_{2,3}\\
		X_{3,1} & X_{3,2} & -1
		\endbmatrix\in\R^{3\times 3}
	\end{equation*}
	cannot be a  feasible solution to \eqref{eq:LP_s2}.
	It follows that $v\eqref{eq:LP_s2}>-2=v\eqref{eq:LP_s2d}$. That is, \eqref{eq:LP_s2d} is not  an exact relaxation  for \eqref{eq:LP_s2},
	so is  the \eqref{eq:LPS-SDP3} relaxation due to
	Lemma  \ref{lem:2}.
\end{example}

We can apply Theorem \ref{th:1} to reveal the hidden convexity of \eqref{minmax0} under certain conditions. The detailed proof is provided in the Appendix.
\begin{corollary}\label{eq:cor}
	Suppose $\mathcal{F}\eqref{eq:LPS-SDP3}\neq \emptyset$.
	When $p\le n-k-m+1$,  problem \eqref{minmax0} admits the following tight SDP relaxation
	\begin{equation}\label{eq:mm:h1}
		\begin{array}{cl}
			\min\limits_{Y\in\S^{n+p},\,t\in\R} &t\\
			{\rm s.t.} & \TR(M_iY_{[1:n][n+1:n+p]})+c_i\le t,~i=1, \cdots, m,\\
			&Y\in\mathcal{F}\eqref{eq:LPS-SDP3}.
		\end{array}
	\end{equation}
Especially, both the Stiefel-constrained weighted maximin dispersion problem \eqref{eq:disp} with $p\le n-m+1$ and the $\ell_1$-norm minimization problem \eqref{eq:orth} with $(m+1)p\le n+1$ have their exact SDP relaxations.
\end{corollary}
As a byproduct of Theorem \ref{th:1}, we characterize the convexity result of a joint numerical range.
\begin{theorem}\label{eq:con}
	If $p\le n-k$, the set
\begin{equation}\label{jn:1}
G_1=\{[\TR(A_1X),\ \cdots,\ \TR(A_kX)]: ~X^TX=I_p,~X\in\R^{n\times p}\}
\end{equation}
is convex, and is equal to
\begin{equation*}G_2=\{[\TR(A_1X),\ \cdots,\ \TR(A_kX)]:  ~X^TX\preceq I_p,~X\in\R^{n\times p}\}.
\end{equation*}
\end{theorem}
\begin{proof}
It is easy to see $G_1\subseteq G_2$ and $G_2$ is convex. It suffices to show $G_2\subseteq G_1$.
	For any $a\in G_2$, we have
	\begin{equation*}
		\mathcal{X}_a:=\{X\in\R^{n\times p}: ~X^TX\preceq I_p,\ \TR(A_iX) = a_i, ~i=1,\cdots,k\}\neq\emptyset.
	\end{equation*}
Since $\mathcal{X}_a$ is compact, it must have  an extreme point, denoted as $\bar X\in\R^{n\times p}$.
According to Theorem \ref{th:1}, Remark \ref{th:rem:1}, and the assumption $p\le n-k$,  $\bar X$  satisfies
	$$\bar X^T\bar X= I_p,\ \TR(A_i\bar X) = a_i, ~i=1,\cdots,k.$$
	That is, $a\in G_1$. We complete the proof.
\end{proof}

\begin{remark}
	Beck \cite{Be07} established convexity of the following joint numerical range
	\begin{equation}\label{Be:1}
		\{[x^Tx+a_0^Tx,\ a_1^Tx,\ \cdots,\ a_k^Tx]:~ x\in\R^n\} \subseteq \R^{ k+1}
	\end{equation}
	where $a_0,\cdots,a_k\in\R^n$ and $k\le n-1$. Up to a translation transformation, we can set $a_0\equiv0$ in \eqref{Be:1} without loss of generalization.  The convexity of  \eqref{Be:1} implies that of $G_1$ \eqref{jn:1} with $p=1$.
\end{remark}

\begin{remark}\label{cor:au}
	Define $B_1, \cdots, B_k$ as in \eqref{eq:Bi}. 
	By \eqref{eq:YX}, \eqref{eq:equi0} and \eqref{eq:62}, $G_1$ \eqref{jn:1} is convex if and only if
	\begin{equation}\label{eq:YXX}
		\begin{aligned}
			\{[\TR(B_1Y)&,\ \cdots,\ \TR(B_kY)]:~Y\in\S^{ n+p },~Y\succeq 0,\  {\rm rank}(Y)=n,\\
			& \TR (E^{i,j}Y)=\delta_{i,j}, ~\forall 1\le i\le j\le n, ~n+1\le i\le j\le n+p\}
		\end{aligned}
	\end{equation}
	is convex. According to Theorem \ref{eq:con},  one can easily verify that \eqref{eq:YXX} is convex if $p\le n-k$.
	Convexity of \eqref{eq:YXX}  can also be obtained from that of
	\begin{equation}\label{eq:conY}
		\begin{aligned}
			\{[\TR(B_1Y),\cdots, \TR(B_kY),&~\TR (E_{i,j}Y)\ ( 1\le i\le j\le n, ~n+1\le i\le j\le n+p)]:\\
			&Y\in\S^{n+p},~Y\succeq 0,~ {\rm rank}(Y)=n\},
		\end{aligned}
	\end{equation}
	which is closely related to the following classical result \cite{Auyeung79}.
	\end{remark}
	\begin{lemma}{\rm\cite[Theorem 3]{Auyeung79}}\label{le:Auyeung}
		If $1\le n_2\le n_1-1$ and $l\le (n_2+2)(n_2+1)/2-1$, then for
		$B_1, \cdots, B_l\in\S^{n_1}$, the set
 \[\{[\TR(B_1Y), \cdots, \TR(B_{l}Y)]: ~Y\in\S^{n_1},\ Y\succeq0,\  {\rm rank}(Y)=n_2\}\]
		is convex, and is equal to
		\[\{[\TR(B_1Y), \cdots, \TR(B_{l}Y)]: ~Y\in\S^{n_1},\ Y\succeq0\}.\]
	\end{lemma}
	In case \eqref{eq:conY}, $n_1=n+p,\ n_2=n,\ l=k+p(p+1)/2+n(n+1)/2$.
	Under the additional assumption \eqref{bd:k1}, we have
	\[
	\begin{aligned}
		&l-( (n_2+1)(n_2+2)/2 -1)\\
		=\ &k+ p(p+1)/2+ n(n+1)/2-(
		(n+1)(n+2)/2-1)\\
		\le\ & n+ n(n+1)/2-
		(n+1)(n+2)/2+1\\
		=\ &0.
	\end{aligned}
	\]
	Therefore,  \eqref{eq:YXX} is convex by Lemma \ref{le:Auyeung}.
	Comparing with Theorem \ref{eq:con}, the condition \eqref{bd:k1} is too restrictive when $p>1$.

\begin{remark}
For the general quadratic function  minimization problem on the Stiefel manifold, Burer and Park \cite{Burer22} strengthened the Shor relaxation, which is equivalent to \eqref{eq:LPS-SDP3} when applying to \eqref{eq:LP_s}, see \cite{Ding11}.
However, Burer and Park's relaxation produces the same inexact solutions as \eqref{eq:LPS-SDP3}
for Examples \ref{exam:ex1}, \ref{exam:ex2}, and \ref{exam:ex3}.
Hidden convexity of the \eqref{eq:LP_s} case with $p-1= n-k$ remains unknown.
\end{remark}		
\begin{remark}
Certain special cases of \eqref{eq:LP_s} with a large number of linear constraints may still exhibit hidden convexity, as illustrated by the linear sum assignment problem (Example \ref{ex:linear_sum}) and the linear bottleneck assignment problem \eqref{exam:6}. Exploring the specific properties of data in the linear constraints could potentially lead to further extensions of the hidden convexity result.
\end{remark}

\section{Sufficient global optimality conditions based on local information}\label{se:5}

In this section, we regard \eqref{eq:LP_s} as an NLP problem, and focus on studying whether  local necessary optimality conditions could guarantee global optimality.

We have revealed hidden convexity of \eqref{eq:LP_s} when $p\le n-k$  in previous sections. However, as shown by the following example, it does not mean that any local minimizer of \eqref{eq:LP_s} is globally optimal.
\begin{example}\label{exm:suff}
	Consider an instance of \eqref{eq:LP_s} with $p=1$, $n=2$ and $k=1$:
		\begin{equation*}
		\begin{array}{cl}
			\min\limits_{x\in St_{2, 1}} &-x_1-2x_2\\
			{\rm s.t.}& x_{1}\le 0.
		\end{array}
	\end{equation*}
	It is not difficult to verify that $(0, 1)^T$ is the unique global minimizer, and $(0, -1)^T$ is a local nonglobal minimizer.
\end{example}


As a well-known result, if the KKT condition holds and the corresponding Lagrangian function is convex with respect to the primal variables, then the global optimality holds. For completeness, we present the corresponding result for \eqref{eq:LP_s} with a detailed proof in the Appendix.
\begin{lemma}\label{le:sd}
	Let $X^* \in E_{n,p,k}$. If there exist $\lambda_i\in\R$ for $i=1, \cdots, k$ and $\Lambda\succeq 0$ such that
	\begin{eqnarray}
		&&A_0^T+\sum_{i=1}^k\lambda_i A_i^T+X^*\Lambda =0,\label{eq:kktY10}\\
		&&\TR (A_i X^*)-a_i\le 0,\  i=1, \cdots, k,\ X^{*T}X^*=I_p,\label{eq:kkt_f0}\\
		&&\lambda_i^+ (\TR (A_i X^*)-a_i^u)=\lambda_i^-(\TR (A_i X^*)-a_i^l)=0,\  i=1, \cdots, k,\label{eq:kktY02}
	\end{eqnarray}
	then $X^*$ is a global minimizer of \eqref{eq:LP_s}.
\end{lemma}

Considering \eqref{eq:LP_s} as an NLP problem in terms of the vector variable in $\R^{np}$,
we obtain that, under certain assumptions, the standard first- and second-order necessary optimality conditions are sufficient for the global minimizer. The proof is presented in the Appendix.
\begin{lemma}\label{le:1}
	Assume $p(p+1)/2+1\le   n-k$. Let $X^*\in E_{n,p,k}$  satisfy LICQ, and the standard first- and second-order necessary optimality conditions. Then $X^*$ is a global minimizer of \eqref{eq:LP_s}.
\end{lemma}

 To our surprise, the above result can be  greatly strengthened by more careful analysis.
\begin{theorem}\label{th:nolocal}
	Suppose $p+1\le n-k$.  Let $X^*\in E_{n,p,k}$  satisfy LICQ, and the first- and second-order necessary optimality conditions. Then,   $X^*$ is a global minimizer of \eqref{eq:LP_s}.
\end{theorem}
\begin{proof}
Let $X^*=[X_1^*, \cdots, X_p^*]$. There exist vectors $X_{p+1}^*, \cdots, X_n^*\in\R^n$ such that $Q:=[X_1^*, \cdots, X_p^*, X_{p+1}^*, \cdots, X_n^*]$ is orthogonal.
Consider the  orthogonal linear transformation $Z=Q^TX$. We have
	\begin{eqnarray}
		&&X^TX=I_p  \Longleftrightarrow Z^TZ=I_p ,\nonumber\\
		&&\TR(A_iX)=\TR(A_iQZ),\  i=0, \cdots, k.\nonumber
	\end{eqnarray}
Define $\tilde A_i=A_iQ$ for $i=0, \cdots, k$.
Then, $Z^*=Q^TX^*=[e_1, \cdots, e_p]$ is a global minimizer of the following \eqref{eq:LP_s} problem:
\begin{equation*}
	\begin{array}{cl}
		\min\limits_{Z\in St_{n,p}} &\TR(\tilde A_0Z)\\
		{\rm s.t.}& a_i^l \le\TR(\tilde A_iZ)\le a_i^u,~i=1, \cdots, k.
	\end{array}
\end{equation*}
Moreover, under any nonsingular linear transformation, LICQ and the first- and second-order necessary optimality conditions remain invariant.  So, without loss of generality, we simply assume that
	\[
	X^*=[e_1, \cdots, e_p]
	\]
	is a global minimizer of the original \eqref{eq:LP_s}.
	
Define the set of active indices for linear inequality constraints as \begin{equation}\label{eq:A}
		\mathcal{A}=\{i:\ \TR (A_i X^*)-a_i^l=0\ {\rm or}\ \TR (A_i X^*)-a_i^u=0\}.
	\end{equation}According to the second-order necessary optimality condition, for any $V\in\R^{n\times p}$ such that
	\begin{eqnarray}
		&&\TR(A_iV) =0,\ i\in \mathcal{A},\label{eq:1second_v1}\\
		&&V_i^TX^*_j+V_j^TX^*_i=0,\ \forall 1\le i\le j\le p,\label{eq:1second_v2}
	\end{eqnarray}
	it must hold that
	\begin{eqnarray}
		&&\V(V)^T (\Lambda\otimes I_n) \V(V)\ge0.\label{eq:1second_v3}
	\end{eqnarray}
Note that \eqref{eq:1second_v2} is
	\begin{equation}
		V_i^Te_j+  V_j^Te_i=0,\ \forall 1\le i\le j\le p.\label{eq:second_tv20}
	\end{equation}
	Define $U=V_{[1:p][1:p]}\in\R^{p\times p}$  and $W=V_{[p+1:n][1:p]}\in\R^{(n-p)\times p}$. That is,
	$V=[U; W].$ Then equations \eqref{eq:second_tv20} are equivalent to
	\begin{equation*}
		U_{ij}+ U_{ji}=0,\ \forall1\le i\le j\le p,
	\end{equation*}
	i.e., $U$ is skew-symmetric.

	In sum, by collecting all strictly upper triangular entries of the skew-symmetric matrix $U$, $U_{ij}$ $(1\le i<j\le n)$,
	as $u\in\R^{p(p-1)/2}$, we reorganize the second-order optimality condition \eqref{eq:1second_v1}-\eqref{eq:1second_v3} as
 \begin{equation}\label{eq:ww}
		[u^T,\ \V(W)^T]b_i=0,\ i\in \mathcal{A} \Longrightarrow [u^T,\ \V(W)^T]M\left[\begin{matrix}u\\\V(W)\end{matrix}\right]\succeq0,
	\end{equation}
	where $b_i$ and $M\in\S^{p(p-1)/2+(n-p)p}$ are rearranged counterparts of $ A_i$ and $\Lambda\otimes I_n$, respectively.
	
	Since $\mathcal{A}$ contains at most $k$ indices and we have assumed $k\le n-p-1$, it follows from \eqref{eq:ww}  that $M$ has at most $n-p-1$ negative eigenvalues.
	On the other hand,
	we have
	\begin{equation}\label{eq:MM}
		M_{[p(p-1)/2+1:p(p-1)/2+(n-p)p]}={(\Lambda\otimes I_n)}_{[p^2+1:p^2+(n-p)p]}=\Lambda\otimes I_{n-p}.
	\end{equation}
Thus, if $\Lambda$ has a negative eigenvalue, then $\Lambda\otimes I_{n-p}$ has at least $n-p$
	negative eigenvalues and so does $M_{[p(p-1)/2+1:p(p-1)/2+(n-p)p]}$ by \eqref{eq:MM}.
	According to Cauchy's interlace theorem \cite{Hw04}, $M$ has at least $n-p$ negative eigenvalues.
	The contradiction implies that $\Lambda\succeq 0$. Then,  according to Lemma \ref{le:sd}, we complete the proof.
\end{proof}

\begin{remark}
	According to Example \ref{exm:suff}, the condition $p+1\le n-k$ assumed in Theorem \ref{th:nolocal} is tight.
\end{remark}

Based on Theorem \ref{th:nolocal},  we immediately obtain a sufficient condition under which \eqref{eq:LP_s} has no local nonglobal minimizer.
\begin{corollary}\label{cor:cond}
If $p+1\le n-k$, and LICQ holds at any feasible point of \eqref{eq:LP_s},
	then any local minimizer of \eqref{eq:LP_s} is globally optimal.
\end{corollary}

The LICQ condition presented in Corollary \ref{cor:cond} guarantees that $E_{n,p,k}$ is an embedded submanifold in $St_{n, p}$. To achieve the same goal, a more generalized condition is assumed that the Jacobin matrices of the constraints at all feasible points have the same rank \cite[Theorem 5.12]{Lee09}.
However, as illustrated by the following example,
the generalized condition does not hold true for general \eqref{eq:LP_s}.
\begin{example}
		Consider an instance of \eqref{eq:LP_s} with $p=2$, $n=4$ and $k=1$:
		\begin{equation*}
			\begin{array}{cl}
				\min\limits_{X\in St_{4,2}} &0\\
				{\rm s.t.}& \TR(AX)=1,
			\end{array}
		\end{equation*}
where
\begin{equation*}
A
	=
	\bmatrix
	1 & 1 & 0 & 0\\
	1 & 0 & 0 & 0
	\endbmatrix.
\end{equation*}
Given the following two feasible points
\begin{equation*}
	\bar X=\bmatrix
	\bar X_1 & \bar X_2
	\endbmatrix
	=
	\bmatrix
	1 & 0\\
	0 & 1\\
	0 & 0\\
	0 & 0
	\endbmatrix,\
	\tilde X=\bmatrix
    \tilde X_1 & \tilde X_2
	\endbmatrix
	=
	\bmatrix
	\frac{1}{2} & 0\\
	\frac{1}{2} & \frac{\sqrt{2}}{2}\\
	0 & \frac{1}{2}\\
	\frac{\sqrt{2}}{2} & -\frac{1}{2}
	\endbmatrix,
\end{equation*}
the Jacobin matrices of the constraints are
\begin{eqnarray*}
&&	\bar J=\bmatrix
	\bar X_1 & 0 &  \bar X_2 &  {(A^T)}_1\\
	0 & \bar X_2 & \bar X_1   &  {(A^T)}_2
	\endbmatrix
	=
	\bmatrix
	1 & 0 & 0 & 1\\
	0 & 0 & 1 & 1\\
	0 & 0 & 0 & 0\\
	0 & 0 & 0 & 0\\
	0 & 0 & 1 & 1\\
	0 & 1 & 0 & 0\\
	0 & 0 & 0 & 0\\
	0 & 0 & 0 & 0
	\endbmatrix, \\
\end{eqnarray*}
and
\begin{eqnarray*}
&&
\tilde J=\bmatrix
\tilde X_1 & 0 &  \tilde X_2 &  {A^T}_1\\
0 & \tilde X_2 & \tilde X_1   &  {A^T}_2
\endbmatrix
=
\bmatrix
\frac{1}{2}  & 0 & 0 & 1\\
\frac{1}{2}  & 0 & \frac{\sqrt{2}}{2} & 1\\
0 & 0 & \frac{1}{2} & 0\\
\frac{\sqrt{2}}{2}  & 0 & -\frac{1}{2} & 0\\
0 & 0 & \frac{1}{2} & 1\\
0 & \frac{\sqrt{2}}{2}  & \frac{1}{2} & 0\\
0 & \frac{1}{2} & 0 & 0\\
0 & -\frac{1}{2} & \frac{\sqrt{2}}{2} & 0
\endbmatrix,
\end{eqnarray*}
respectively.
One can easily verify that rank$(\bar J)=3$ and rank$(\tilde J)=4$.
\end{example}

\begin{remark}
There are two trivial cases of \eqref{eq:LP_s} where the feasible sets are manifolds: (i) $k=0$; (ii) $p=1$ and all linear constraints are equality ones.
In case (i), the feasible set is $St_{n,p}$, and so LICQ holds at any feasible point. In case (ii), the feasible set, if not empty, is either a singleton or a sphere.
\end{remark}

\section{Conclusion}
We have introduced a new optimization model \eqref{eq:LP_s} by adding $k$ linear constraints to the linear optimization on the Stiefel manifold $St_{n,p}$. 
We provided some applications of \eqref{eq:LP_s}, from which we showed that \eqref{eq:LP_s} is NP-hard in general. 
Moreover, \eqref{eq:LP_s} also serves as the key subproblem in extended algorithms to 
solve constrained optimization problems with an additional Stiefel manifold constraint.
Following the slightly extended Shapiro-Barvinok-Pataki theorem, 
we verified that \eqref{eq:LP_s} admits an exact  SDP relaxation when $p(p+1)/2\le n-k$. 
Our analysis showed that this sufficient condition can be greatly strengthened to $p \le n-k$, 
which includes the classical case $p\le n$ and $k=0$. 
The hidden convexity, particularly in the case of $p-1=n-k$, remains unknown. It should be noted that hidden convexity does not mean that any local minimizer of \eqref{eq:LP_s} is globally optimal. Regarding  \eqref{eq:LP_s} as a smooth NLP problem, we showed that, at a feasible point satisfying LICQ, the standard first- and second-order necessary optimality conditions are sufficient for the global optimality when $p+1\le n-k$. When and only when LICQ holds at all feasible points of \eqref{eq:LP_s} is unknown. Another theoretical question is when and only when the intersection of a hyperplane and $St_{n,p}$ remains a manifold.

\section*{Appendix}
\subsection* {Proof of Proposition \ref{pro:nphard}}
Given any integer vector $a\in\R^n$, the NP-complete partitioning problem
asks whether
\begin{equation}\label{pp0}\tag{PP}
	a^Ty=0,\ y\in{\{-1, 1\}}^n
\end{equation}
has a solution.
According to \eqref{reform:0},
\eqref{pp0} is feasible if and only if problem
\begin{equation}\label{eq:1}
	\begin{array}{cl}
		\min\limits_{X\in St_{n,1}} & a^TX\\
		{\rm s.t.} & a^TX\ge 0,\\
		&  -\frac{1}{\sqrt{n}}\le e_i^TX\le \frac{1}{\sqrt{n}},~i=1, \cdots, n
	\end{array}
\end{equation}
has an optimal value of $0$. Problem \eqref{eq:1} can be expanded to
the following \eqref{eq:LP_s} with $k= n+1$  and any $p\ge1$:
\begin{equation*}
	\begin{array}{cl}
		\min\limits_{X\in St_{n,p}} & \TR(A_0X)\\
		{\rm s.t.} & \TR(A_0X)\ge0,\\
		&  -\frac{1}{\sqrt{n}}\le \TR(E_iX) \le \frac{1}{\sqrt{n}},~i=1, \cdots, n,
	\end{array}
\end{equation*}
where $A_0^T=[a~ 0~\cdots 0]\in\R^{n\times p}$ and $E_i^T=[e_i~ 0~\cdots 0]\in\R^{n\times p}$.
We complete the proof by the fact that \eqref{pp0} is reduced in polynomial time to an \eqref{eq:LP_s} case.

\subsection*{Proof of Lemma \ref{lem:2}}
	By \eqref{eq:YX}, \eqref{eq:equi0}, and \eqref{eq:Bi},
	$\mathcal{F}\eqref{eq:LPS-SDP3}\neq \emptyset$ if and only if
	$\mathcal{F}\eqref{eq:LPS-R}\neq \emptyset$.
	Let $X$ be any feasible solution of \eqref{eq:LPS-R}. Define $Y$  as in \eqref{eq:YX}.
	Then $Y$ is a feasible solution of \eqref{eq:LPS-SDP3} with the same objective function value as $X$ for \eqref{eq:LPS-R}. The converse is also true. Consequently, $v\eqref{eq:LPS-R}=v\eqref{eq:LPS-SDP3}<+\infty$.
	Besides, according to $X^TX\preceq I_p$, we have
	\[
	X_i^TX_i\le 1,~i=1,\cdots,p.
	\]
	Thus $\mathcal{F}\eqref{eq:LPS-R}$ is compact.  Hence, $\eqref{eq:LPS-R}$ has an optimal solution and so does
	\eqref{eq:LPS-SDP3}.

\subsection*{Proof of Proposition \ref{prop:1}}
	Let $Y$ be any feasible solution of \eqref{eq:LPS-SDP3}. Rewrite $Y$ as in \eqref{eq:YX}.
	According to \eqref{eq:YxX},  we have
	\begin{equation*}
		{\rm rank}(Y)={\rm rank}(I_n)+{\rm rank}(I_p-X^TX)
		= n+{\rm rank}(I_p-X^TX),
	\end{equation*}
	which implies that
	\begin{eqnarray}
		&&{\rm rank}(Y)\ge n,\label{eq:equi01}\\
		&&{\rm rank}(Y)=n \Longleftrightarrow  X^TX=I_p. \label{eq:equi1}
	\end{eqnarray}
	Let $Y^*$ be an optimal solution of \eqref{eq:LPS-SDP3} and ${\rm rank}(Y^*)\le n$. It follows from  \eqref{eq:equi01} that  ${\rm rank}(Y^*)=n$.
	According to Lemma \ref{lem:2}, $X^*=Y^*_{[1:n][n+1:n+p]}$ is an optimal solution of \eqref{eq:LPS-R}.
	Moreover, $X^*$ remains feasible to \eqref{eq:LP_s} by \eqref{eq:equi1}. Since $v$\eqref{eq:LPS-R} $\le v$\eqref{eq:LP_s}, we conclude that  $v\eqref{eq:LPS-R}=v\eqref{eq:LP_s}$ and then complete the proof by Lemma \ref{lem:2}.
\subsection*{Proof of Proposition \ref{le:beck09}}
	The number of linear constraints of  $\eqref{eq:LPS-SDP3}$ is
	\[
	m=n(n+1)/2+p(p+1)/2+k.
	\]
	By assumption \eqref{bd:k1}, we have
	$$
	m\le (n+1)(n+2)/2- 1.$$
	According to Lemma \ref{le:pataki}, ${\rm rank}(Y^e)\le n$ holds for all $Y^e\in \mathcal{E}(\mathcal{F}\eqref{eq:LPS-SDP3})$. Since the objective function is linear, \eqref{eq:LPS-SDP3} has an optimal solution in $ \mathcal{E}(\mathcal{F}\eqref{eq:LPS-SDP3})$. That is,  there exists at least one optimal solution of \eqref{eq:LPS-SDP3} whose rank is at most $n$.  By Proposition \ref{prop:1}, we complete the proof.
\subsection*{Proof of Corollary \ref{eq:cor}}
	Clearly, \eqref{eq:mm:h1} is an SDP relaxation for \eqref{minmax0} under transformation \eqref{eq:YX}. Suppose $\mathcal{F}\eqref{eq:LPS-SDP3}\neq \emptyset$, we have $\mathcal{F}\eqref{eq:mm:h1}\neq \emptyset$ by selecting $t$ large enough. Based on a proof similar to that of Lemma \ref{lem:2}, one can show that \eqref{eq:mm:h1} has an optimal solution. According to \eqref{eq:equi1}, it suffices to show that there exists an optimal solution of rank $n$.
	Let $(Y^*, t^*) \in\S^{n+p}\times\R$ be a global minimizer of \eqref{eq:mm:h1}. There is at least one index $q\in \{1, \cdots, m\}$ such that
	\[
	\TR(M_{q}Y^*_{[1:n][n+1:n+p]})+c_{q}= t^*\ge \TR(M_{i}Y^*_{[1:n][n+1:n+p]})+c_{i},~\forall\, i\in\{1, \cdots, m\}\setminus\{q\}.
	\]
	Then, $Y^*$ remains  an optimal solution of the following SDP problem
	\begin{equation}\label{eq:mm:h2}
		\begin{array}{cl}
			\min &
			\TR(M_qY_{[1:n][n+1:n+p]}) \\
			{\rm s.t.} &\TR(M_qY_{[1:n][n+1:n+p]})+c_q\ge \TR(M_jY_{[1:n][n+1:n+p]})+c_j,~\forall j\neq q,\\
			&Y\in \mathcal{F}\eqref{eq:LPS-SDP3},
		\end{array}
	\end{equation}
	which is a special  case of  \eqref{eq:LPS-SDP3} with $k+m-1$ linear constraints.
	According to Theorem \ref{th:1} and Remark \ref{th:rem:1}, under the assumption $p\le n-(k+m-1)$,  \eqref{eq:mm:h2} has an optimal solution of rank $n$. Since $v\eqref{eq:mm:h1}=v$\eqref{eq:mm:h2}, the rank-$n$ solution of \eqref{eq:mm:h2} is also optimal to \eqref{eq:mm:h1}. Thus, \eqref{eq:mm:h1} must have a rank-$n$ solution, which implies that \eqref{eq:mm:h1} is exact for \eqref{minmax0}.
	As presented by Example \ref{exam:6}, \eqref{eq:disp} and \eqref{eq:orth} are special cases of \eqref{minmax0}, then one can easily derive their specific exactness conditions.
\subsection*{Proof of Lemma \ref{le:sd}}
	For any $X \in E_{n,p,k}$, we have
	{\small	\begin{eqnarray}
			&&\TR(A_0X)\nonumber\\
			\ge\ &&\TR(A_0X)+\sum_{i=1}^k\left[\lambda_i^+ (\TR(A_iX)-a_i^u)-\lambda_i^- (\TR(A_iX)-a_i^l)\right]+\TR(\Lambda(X^{T}X-I_p))\label{eq:s_v0}\\
			=\ &&\TR(A_0X)+\sum_{i=1}^k\lambda_i \TR(A_iX)+\TR(\Lambda(X^{T}X-I_p))-\sum_{i=1}^k(\lambda_i^+a_i^u-\lambda_i^-a_i^l)\label{eq:s_v01}\\
			\ge\ &&\TR(A_0X^*)+\sum_{i=1}^k\lambda_i \TR(A_iX^*)+\TR(\Lambda(X^{*T}X^*-I_p))-\sum_{i=1}^k(\lambda_i^+a_i^u-\lambda_i^-a_i^l)\label{eq:s_v1}\\
			=\ &&\TR(A_0X^*)+\sum_{i=1}^k\left[\lambda_i^+ (\TR(A_iX^*)-a_i^u)-\lambda_i^- (\TR(A_iX^*)-a_i^l)\right]+\TR(\Lambda(X^{*T}X^*-I_p))\nonumber\\
			=\ &&\TR(A_0X^*),\label{eq:s_v2}
	\end{eqnarray}}
	where \eqref{eq:s_v0} holds due to the feasibility of $X$, inequality \eqref{eq:s_v1} holds since $X^*$  globally minimizes the unconstrained convex quadratic programming \eqref{eq:s_v01} as $\Lambda\succeq 0$ and
	\eqref{eq:kktY10}, the last equality \eqref{eq:s_v2} holds due to \eqref{eq:kkt_f0} and \eqref{eq:kktY02}.
	Therefore, $X^*$ is a global minimizer of \eqref{eq:LP_s}.
\subsection*{Proof of Lemma \ref{le:1}}
	Under the assumptions,   there exists $\lambda_i\in\R$ for $i=1, \cdots, k$ and $\Lambda\in\S^{p}$ such that KKT condition \eqref{eq:kktY10}-\eqref{eq:kktY02} holds.
	According to the second-order necessary optimality condition, for any $V\in\R^{n\times p}$ satisfying \eqref{eq:1second_v1} and \eqref{eq:1second_v2}, \eqref{eq:1second_v3} must hold.
	Note that $V\in\R^{n\times p}$ has $np$ entries and \eqref{eq:1second_v1}-\eqref{eq:1second_v2} contains at most  $k+p(p+1)/2$ linear constraints.
	Under assumption $p(p+1)/2+1\le n-k$, we have
	$$np-k- p(p+1)/2 \ge np-(n-1).$$
	Therefore, the set of $V$ satisfying \eqref{eq:1second_v1}-\eqref{eq:1second_v2} is a subspace with dimension no less than $np-(n-1).$
	Then, according to \eqref{eq:1second_v3}, $\Lambda\otimes I_n\in\R^{np\times np}$ has at most $n-1$ negative eigenvalues.
	On the other hand, if $\Lambda$ has at least one negative eigenvalue, then
	$\Lambda\otimes I_n$ has at least $n$ negative eigenvalues.  Thus, it follows that $\Lambda\succeq 0$. According to Lemma \ref{le:sd}, $X^*$ is  a global minimizer of \eqref{eq:LP_s}.
	The proof is complete.

\renewcommand\theequation{A.\arabic{equation}}
\renewcommand\thelemma{A.\arabic{lemma}}
\renewcommand\thetheorem{A.\arabic{theorem}}

\bibliographystyle{siamplain}
\bibliography{references}

\end{document}